\documentclass[11pt,reqno]{amsart}
\usepackage[T1]{fontenc}
\usepackage[utf8]{inputenc}
\usepackage{lmodern}
\usepackage{amsmath,amssymb,amsthm,mathtools}
\usepackage[margin=1in]{geometry}
\usepackage[pagebackref=true,bookmarksopen=true]{hyperref}
\usepackage[en-US]{datetime2} \DTMsetdatestyle{iso} \date{\today}

\newtheorem{theorem}{Theorem}
\newcommand{\dd}{\,\mathrm d}

\title{A sharp lower bound for torsional rigidity of convex domains}
\author{M.A.~Perelmuter}
\address{SCAD Soft Ltd., 3a Osvity Street, Kyiv, 03037, Ukraine}
\email{mikeperelmuter@gmail.com}

\begin{document}
\begin{abstract}
For every bounded convex domain $\Omega\subset\mathbb R^d$, $d\ge2$, we establish the following sharp inequality for torsional rigidity
$$\frac{2(d+1)}{d}\int\limits_\Omega\operatorname{dist}(x,\partial\Omega)^2\,\dd x\le T(\Omega),$$
where $T(\Omega)=4\int\limits_\Omega u\,\dd x$ and $-\Delta u=1$ in $\Omega$, $u\in H_0^1(\Omega)$. Moreover, equality holds if and only if $\Omega$ is a ball.
\end{abstract}
\maketitle

\section{Introduction}
Let $\Omega\subset\mathbb R^d$, $d\ge2$, be a bounded domain, and let $u\in H_0^1(\Omega)$ solve
\begin{equation}
-\Delta u=1\quad\text{in }\Omega.
\end{equation}
The \emph{torsional rigidity} of $\Omega$, in the normalization used here, is
$$T(\Omega)=4\int\limits_\Omega u\dd x.$$
Equivalently, $T(\Omega)$ admits the following two variational characterizations:
\begin{equation}\label{eq:variational}
T(\Omega)=4\sup_{\substack{\psi\in H_0^1(\Omega)\\\psi\ne0}}
\frac{\left(\int\limits_\Omega|\psi|\dd x\right)^2}
{\int\limits_\Omega|\nabla\psi|^2\dd x}=4\sup_{\psi\in H_0^1(\Omega)}
\left\{2\int\limits_\Omega\psi\dd x-\int\limits_\Omega|\nabla\psi|^2\dd x\right\}.
\end{equation}
For planar domains, torsional rigidity is a classical quantity in the theory of elasticity. Its estimation in terms of geometric characteristics of the domain has a long history (see the monograph of P\'olya and Szeg\H{o}~\cite{PolyaSzego}).

A natural geometric quantity in this context is the second moment of the distance to the boundary,
$$I_2(\Omega)=\int\limits_\Omega\delta(x)^2\dd x,\qquad \delta(x)=\operatorname{dist}(x,\partial\Omega).$$

For simply connected plane domains, Avkhadiev~\cite{Avkhadiev1998} established the two-sided comparison
\begin{equation}\label{eq:Avkhadiev}
I_2(\Omega)\le T(\Omega)\le64I_2(\Omega).
\end{equation}
An attractive feature of these estimates is their two-sided nature: they establish the equivalence of torsional rigidity and the geometric quantity $I_2(\Omega)$. Two-sided estimates involving the conformal radius were also obtained by Avkhadiev~\cite{Avkhadiev1998}. However, the classical conformal radius is a planar notion whose computation generally requires a conformal mapping, whereas $I_2(\Omega)$ depends only on the Euclidean distance to the boundary and is therefore more readily accessible.
For nonconvex simply connected planar domains, it is not known whether the constant $64$ in Avkhadiev's upper bound is optimal. Numerical computations by Abramov~\cite{Abramov2011} suggest that the optimal universal constant $\lambda$ in $T(\Omega)\le\lambda I_2(\Omega)$ is at least $6$.
Salakhudinov~\cite{Sal2001} subsequently improved the lower bound to
$$\frac32 I_2(\Omega)\le T(\Omega).$$

Ba\~nuelos, van den Berg and Carroll~\cite{BanuelosVanDenBergCarroll2002} studied comparisons between torsional rigidity and $I_2$ in general domains in $\mathbb{R}^d$. In particular, they proved the lower bound $(2/d)I_2(\Omega)\le T(\Omega)$ and obtained upper bounds under additional assumptions. For convex plane domains, the classical Makai inequality~\cite{Makai1962} gives the sharper upper estimate $T(\Omega)\le4I_2(\Omega)$ (see also~\cite{Sal2016,Sal2018}).

In~\cite{Sal2018}, Salakhudinov formulated the following \emph{conjecture}
\begin{equation}\label{eq:conjecture}
3I_2(\Omega)\le T(\Omega)
\end{equation}
for simply connected plane domains. He proved it for certain special convex domains, including polygons circumscribed about a circle.

Related geometric inequalities for convex domains and their boundary-distance moments were obtained in~\cite{Sal2012,Sal2016,GafSal2021,GafSal2024}. Higher-dimensional torsional-rigidity inequalities have also been studied (see~\cite{Avkhadiev2012}).

\bigskip
The purpose of this note is to prove inequality~\eqref{eq:main} below for all bounded convex domains in $\mathbb R^d$, $d\ge2$. For $d=2$ it gives~\eqref{eq:conjecture} for all bounded convex plane domains.

The proof is based on the variational characterization of torsional rigidity, which allows one to obtain a lower bound by choosing a single suitable test function. We use a test function depending only on the distance to the boundary, reducing the problem to a one-dimensional estimate. This is combined with the concavity of the $d$-th root of the volume of inner parallel sets, which follows from the Brunn--Minkowski inequality.

\begin{theorem}\label{thm:main}
Let $\Omega\subset\mathbb R^d$, $d\ge2$, be a bounded open convex set. Then
\begin{equation}\label{eq:main}
\frac{2(d+1)}{d}I_2(\Omega)\le T(\Omega).
\end{equation}
Equality holds if and only if $\Omega$ is a ball; in particular, the constant is optimal.
\end{theorem}

Together with the upper bound of Makai in the plane~\cite[formula~(4)]{Makai1962} and its extension to all dimensions by Prinari and Zagati~\cite[Theorem~1.1 with $p=2$, $q=1$]{PrinariZagati}, Theorem~\ref{thm:main} yields, for every bounded open convex set $\Omega\subset\mathbb R^d$, $d\ge2$, the two-sided estimate
\begin{equation}\label{eq:twosided}
\frac{2(d+1)}{d}\,I_2(\Omega)\le T(\Omega)\le 4I_2(\Omega).
\end{equation}
In particular, $3I_2(\Omega)\le T(\Omega)\le4I_2(\Omega)$ for $d=2$. Both constants are sharp: the lower one is attained for balls, and the upper one is approached by the thin flat domains $\Omega_L=(-L/2,L/2)^{d-1}\times(0,1)$ as $L\to\infty$ (for $d=2$, long thin rectangles). We note that in~\cite{PrinariZagati} the torsional rigidity is normalized as $\int\limits_\Omega u\dd x$, so their bound $\int\limits_\Omega u\dd x\le I_2(\Omega)$ reads $T(\Omega)\le4I_2(\Omega)$ in the normalization used here.

For the lower bound, equality holds if and only if $\Omega$ is a ball, as proved in Section~\ref{sec:equality}.

\section{Proof of the theorem}
\begin{proof}
Let $\rho=\max\limits_{x\in\Omega}\delta(x)$ and
$$\Omega_t=\{x\in\Omega:\delta(x)>t\},\quad V(t)=|\Omega_t|,\quad L(t)=-V'(t)\quad\text{a.e.}$$
Since $\Omega$ is convex, $\delta$ is concave~\cite[Theorem~2.3.2]{BalinskyEvansLewis}. Hence, for $x\in\Omega_t$, $y\in\Omega_s$ and $\lambda\in[0,1]$, we have $\delta((1-\lambda)x+\lambda y)\ge(1-\lambda)\delta(x)+\lambda\delta(y)>(1-\lambda)t+\lambda s$, that is,
$\Omega_{(1-\lambda)t+\lambda s}\supseteq(1-\lambda)\Omega_t+\lambda\Omega_s$. By the Brunn--Minkowski inequality~\cite[Theorem~7.1.1]{Schneider}, the function $f(t)=V(t)^{1/d}$ is concave on $[0,\rho)$. Since $|\nabla\delta|=1$ almost everywhere~\cite[Theorem~2.2.7]{BalinskyEvansLewis}, the coarea formula~\cite[Theorem~3.10]{EvansGariepy} gives $V(t)=\int\limits_t^\rho\mathcal H^{d-1}(\{\delta=s\})\,\dd s$. Thus $V$ is absolutely continuous on $[0,\rho]$ and $L(t)=\mathcal H^{d-1}(\{\delta=t\})$ for almost every $t$, where $\mathcal H^{d-1}$ denotes the $(d-1)$-dimensional Hausdorff measure. Moreover, for $0\le t<s<\rho$, continuity of $\delta$ on the connected set $\Omega$ and its range $(0,\rho]$ imply that $\{t<\delta<s\}$ is nonempty and open. Hence $\Omega_t\setminus\Omega_s$ contains a nonempty open set and $V(t)>V(s)$. Thus $f$ is strictly decreasing, and concavity gives $f'(t)<0$ wherever $f'$ exists. In particular, $L(t)>0$ almost everywhere on $(0,\rho)$.

Let $\varphi:[0,\rho]\to\mathbb R$ be a Lipschitz function with $\varphi(0)=0$. Take $\psi(x)=\varphi(\delta(x))$. It is Lipschitz and vanishes on $\partial\Omega$, so $\psi\in H_0^1(\Omega)$, and $|\nabla\psi|=|\varphi'(\delta)|$ almost everywhere.\footnote{If $N\subset[0,\rho]$ is the null set where $\varphi'$ is undefined, then $|\delta^{-1}(N)|=\int\limits_N L\dd t=0$ by the coarea formula, so $\varphi'(\delta)$ is defined almost everywhere in $\Omega$.} By the second variational characterization in~\eqref{eq:variational},
$$ 4\left(2\int\limits_\Omega\psi(x)\dd x -\int\limits_\Omega|\nabla\psi(x)|^2\dd x\right)\le T(\Omega).$$
Since $\psi(x)=\varphi(\delta(x))$, the coarea formula and
$|\nabla\delta|=1$ almost everywhere give
$$\int\limits_\Omega\psi(x)\dd x =\int\limits_0^\rho\varphi(t)L(t)\dd t$$
and
$$\int\limits_\Omega|\nabla\psi(x)|^2\dd x=\int\limits_0^\rho\varphi'(t)^2L(t)\dd t.$$
Consequently,
\begin{equation}\label{eq:phi}
4\int\limits_0^\rho\bigl(2\varphi(t)-\varphi'(t)^2\bigr)L(t)\dd t\le T(\Omega).
\end{equation}
Using $V(t)=\int\limits_t^\rho L(s)\dd s$ and integration by parts, we obtain
$$\int\limits_0^\rho2\varphi L\dd t=2\int\limits_0^\rho\varphi'V\dd t.$$
Consequently, \eqref{eq:phi} yields, for every such $\varphi$,
$$4\int\limits_0^\rho\bigl(2\varphi'V-\varphi'^2L\bigr)\dd t\le T(\Omega).$$

For almost every $t\in(0,\rho)$, the quadratic expression $2qV(t)-q^2L(t)$ is maximized at $q=V(t)/L(t)$. We therefore set
$\varphi_M'(t)=\min\{V(t)/L(t),M\}$, with $\varphi_M(0)=0$.\footnote{The formal maximizing choice $\varphi'=V/L$ need not be bounded. Truncation yields admissible Lipschitz test functions, while letting $M\to\infty$ recovers the maximal value of the quadratic expression.}
For each $M>0$, $\varphi_M$ is Lipschitz, so $\varphi_M(\delta)\in H_0^1(\Omega)$.
Applying the preceding bound with $\varphi=\varphi_M$ gives
$$4\int\limits_0^\rho\left(2\varphi_M'(t)V(t)-\varphi_M'(t)^2L(t)\right)\dd t\le T(\Omega).$$
Since $q\mapsto 2qV-q^2L$ is nonnegative and increasing on $0\le q\le V/L$, the integrand $2\varphi_M'V-(\varphi_M')^2L$ increases pointwise to $V^2/L$ as $M\to\infty$.
Monotone convergence therefore yields
\begin{equation}\label{eq:optimal}
4\int\limits_0^\rho\frac{V(t)^2}{L(t)}\dd t\le T(\Omega).
\end{equation}
At almost every $t\in(0,\rho)$, set $a(t)=-f'(t)>0$. By concavity,
$$0\le f(s)\le f(t)-a(t)(s-t),\qquad t\le s<\rho.$$
In particular, $\rho-t\le f(t)/a(t)$. With the substitution $z=f(t)-a(t)r$, we obtain
\begin{align*}
\int\limits_t^\rho V(s)\dd s \le \int\limits_0^{f(t)/a(t)}(f(t)-a(t)r)^d\dd r
=\frac1{a(t)}\int\limits_0^{f(t)}z^d\dd z =\frac{f(t)^{d+1}}{(d+1)a(t)}.
\end{align*}
Since $L(t)=d f(t)^{d-1}a(t)$, it follows that
\begin{equation}\label{eq:pointwise}
\frac{d+1}{d}\int\limits_t^\rho V(s)\dd s\le\frac{f(t)^{d+1}}{d\,a(t)}=\frac{V(t)^2}{L(t)}.
\end{equation}
Since
$$\delta(x)^2=2\int\limits_0^{\delta(x)}t\dd t,$$
Tonelli's theorem~\cite[Theorem~2.37]{Folland} yields
$$\int\limits_\Omega\delta(x)^2\dd x
=2\int\limits_0^\rho t\,|\{x\in\Omega:\delta(x)>t\}|\dd t
=2\int\limits_0^\rho tV(t)\dd t
=2\int\limits_0^\rho\left(\int\limits_t^\rho V(s)\dd s\right)\dd t.$$
Combining this identity with~\eqref{eq:pointwise} and~\eqref{eq:optimal}, we obtain
$$\frac{2(d+1)}d\int\limits_\Omega\delta(x)^2\dd x\le T(\Omega).$$

Finally, let $\Omega=B_R$ be a ball of radius $R$ and $\omega_d=|B_1|$. The torsion function of $B_R$ is given explicitly by $u(x)=(R^2-|x|^2)/(2d)$. Hence
$$T(B_R)=\frac{4\omega_dR^{d+2}}{d(d+2)},\qquad
I_2(B_R)=\int\limits_{B_R}(R-|x|)^2\dd x=\frac{2\omega_dR^{d+2}}{(d+1)(d+2)},$$
and therefore $T(B_R)=\dfrac{2(d+1)}{d}I_2(B_R)$. Thus equality holds for balls, and the constant is optimal.
\end{proof}

For nonconvex domains the lower bound corresponding to~\eqref{eq:main} remains an \emph{open problem}.

\section{The equality case}\label{sec:equality}
Suppose equality holds in~\eqref{eq:main}. Recall that
$\rho=\max\limits_{x\in\Omega}\delta(x)$.
Indeed, by~\eqref{eq:optimal}, \eqref{eq:pointwise}, and Tonelli's theorem,
$$ T(\Omega)\ge4\int\limits_0^\rho\frac{V(t)^2}{L(t)}\dd t
\ge\frac{4(d+1)}d\int\limits_0^\rho\left(\int\limits_t^\rho V(s)\dd s\right)\dd t
=\frac{2(d+1)}d I_2(\Omega). $$
Since equality holds between the first and last expressions, equality holds in both inequalities. In particular, equality holds in~\eqref{eq:pointwise} for almost every $t\in(0,\rho)$. Consequently,
$$\int\limits_t^\rho f(s)^d\,\dd s =\frac{f(t)^{d+1}}{(d+1)a(t)}\quad\text{for almost every }t\in(0,\rho).$$

\emph{Affineness of $f$ (proof sketch).}
Since $a=-f'>0$ almost everywhere, the preceding identity can be written as
$$\bigl(f^{-d}\bigr)'(t)\int\limits_t^\rho f(s)^d\,\dd s =\frac{d}{d+1}\quad\text{for almost every }t\in(0,\rho).$$
As $f>0$ is locally Lipschitz and the integral is positive, this identity
provides the local regularity needed to differentiate it in the sense of
distributions. Differentiating and simplifying yields
$$f''=0\quad\text{in }\mathcal D'(0,\rho).$$
Thus $f$ is affine. Since $f(t)\to0$ as $t\uparrow\rho$, we obtain
$$f(t)=c(\rho-t),\qquad 0<t<\rho,\quad c>0.$$
Consequently, $a(t)=-f'(t)=c$ almost everywhere. Since $V=f^d$ and $L=-V'=d f^{d-1}a$ almost everywhere,
$$\frac{V(t)}{L(t)}=\frac{f(t)}{d\,a(t)}=\frac{\rho-t}{d} \quad\text{for almost every }t\in(0,\rho).$$

\emph{Identification of the torsion function.}
Let $J(z)=2\int\limits_\Omega z\,\dd x-\int\limits_\Omega|\nabla z|^2\,\dd x$, and let $u\in H_0^1(\Omega)$ be the torsion function. By~\eqref{eq:variational}, we have $4J(u)=T(\Omega)$. Recall that, for $v=\varphi(\delta)$ with $\varphi$ Lipschitz and $\varphi(0)=0$, integration by parts, using $V'=-L$ and $\varphi(0)=0$, gives
$$J(v)=\int\limits_0^\rho\bigl(2\varphi'V-\varphi'^2L\bigr)\dd t.$$
The integrand is largest when $\varphi'=V/L$ (as in the choice of the test function in the proof of the theorem). In the equality case, $V/L=(\rho-t)/d$ by the previous step. Integrating $\varphi'=V/L$ with the initial condition $\varphi(0)=0$, we obtain
$$\varphi(t)=\frac{2\rho t-t^2}{2d},\qquad v=\varphi(\delta).$$
Then $v\in H_0^1(\Omega)$, because $\varphi$ is Lipschitz and $\varphi(0)=0$. Since $\varphi'=V/L$ almost everywhere, we have
$$2\varphi'V-(\varphi')^2L=\frac{V^2}{L},$$
and hence
$$J(v)=\int\limits_0^\rho\frac{V(t)^2}{L(t)}\dd t.$$
Integrating~\eqref{eq:pointwise} over $(0,\rho)$ and using Tonelli's theorem, and then~\eqref{eq:optimal}, we find
$$T(\Omega)=\frac{2(d+1)}d\,I_2(\Omega)=\frac{4(d+1)}d\int\limits_0^\rho tV(t)\dd t\le4\int\limits_0^\rho\frac{V(t)^2}{L(t)}\dd t\le T(\Omega),$$
so $T(\Omega)=4\int_0^\rho V^2/L\,\dd t=4J(v)$. By uniqueness of the torsion function, this function coincides with $u$, that is,
% For uniqueness, see \cite[Section~6.2.1]{EvansPDE}.
\begin{equation}\label{eq:torsion-equality}
u(x)=\frac{2\rho\delta(x)-\delta(x)^2}{2d}.
\end{equation}

\emph{Rigidity of the domain.}
Let $\varphi(t)=\dfrac{2\rho t-t^2}{2d}$. By~\eqref{eq:torsion-equality}, we have $u=\varphi(\delta)$.
Note that $\varphi'(t)=(\rho-t)/d\ge0$ on $[0,\rho]$, hence $\varphi$ is nondecreasing there, and
$$\varphi(\rho-s)=\frac{\rho^2-s^2}{2d},\qquad 0\le s\le\rho.$$
Let $z\in\Omega$ be a point with $\delta(z)=\rho$ and let $B=B_\rho(z)$.
By the definition of $\delta$, $B\subset\Omega$. The torsion function of $B$ is given explicitly by
$$u_B(x)=\frac{\rho^2-|x-z|^2}{2d}=\varphi\bigl(\rho-|x-z|\bigr).$$
We regard this quadratic polynomial as defined on all of $\mathbb R^d$.

For $x\in B$ we have $B_{\rho-|x-z|}(x)\subset B\subset\Omega$, so
$\rho-|x-z|\le\delta(x)\le\rho$. Since $\varphi$ is nondecreasing on $[0,\rho]$,
$$u(x)=\varphi(\delta(x))\ge\varphi\bigl(\rho-|x-z|\bigr)=u_B(x)\qquad(x\in B),$$
with equality at $x=z$, because $u(z)=\varphi(\rho)=\rho^2/(2d)=u_B(z)$.

By interior elliptic regularity, $u$ is smooth in $\Omega$, while $u_B$ is a quadratic polynomial. Thus $h=u-u_B$ is harmonic in $\Omega$. Moreover, $h\ge0$ in $B$ and $h(z)=0$ at an interior point of $B$. By the strong maximum principle, $h\equiv0$ in $B$. Since $\Omega$ is connected, unique continuation for harmonic functions yields $h\equiv0$ in $\Omega$.
Thus $u=u_B$ in $\Omega$. Since $u>0$ in $\Omega$, we have
$$\frac{\rho^2-|x-z|^2}{2d}=u_B(x)=u(x)>0 \qquad (x\in\Omega).$$
Hence $\Omega\subset B_\rho(z)$. Together with $B_\rho(z)\subset\Omega$, this gives $\Omega=B_\rho(z)$.

\section*{Acknowledgments}
AI tools were used during the preparation of this manuscript to assist with proofreading and checking mathematical calculations. The final manuscript was entirely reviewed, revised, and verified by the author, who retains sole responsibility for the final content and mathematical accuracy of this work.

\end{document}